\documentclass[11pt]{amsart}
\usepackage{amsmath,amssymb,amsthm}
\usepackage[colorlinks=true,linkcolor=blue,citecolor=blue,urlcolor=blue]{hyperref}

\newtheorem{theorem}{Theorem}[section]
\newtheorem{lemma}[theorem]{Lemma}
\newtheorem{proposition}[theorem]{Proposition}
\newtheorem{corollary}[theorem]{Corollary}
\theoremstyle{remark}
\newtheorem{remark}[theorem]{Remark}

\numberwithin{equation}{section}

\newcommand{\R}{\mathbb{R}}
\newcommand{\Sn}{\mathbb{S}^{n}}
\newcommand{\bg}{\bar{g}}
\newcommand{\bn}{\bar{\nabla}}
\newcommand{\td}{\widetilde}
\newcommand{\wh}{\widehat}
\newcommand{\cL}{\mathcal{L}}
\newcommand{\cP}{\mathcal{P}}
\newcommand{\cC}{\mathcal{C}}
\newcommand{\cF}{\mathcal{F}}
\newcommand{\cG}{\mathcal{G}}
\newcommand{\norm}[1]{\lVert #1 \rVert}
\newcommand{\abs}[1]{\lvert #1 \rvert}

\title[Supercritical dual $k$-Minkowski flows]{Supercritical dual $k$-Minkowski flows with general prescribed data: \\
a priori estimates and asymptotic convergence}

\author{Hongyi Sheng}
\address{Institute for Theoretical Sciences, Westlake Institute for Advanced Study, Westlake University, Hangzhou 310030, China}
\email{shenghongyi@westlake.edu.cn}

\author{Weimin Sheng}
\address{School of Mathematical Sciences, Zhejiang University, Hangzhou 310058, China}
\email{weimins@zju.edu.cn}

\author{Jiazhuo Yang}
\address{School of Mathematical Sciences, Zhejiang University, Hangzhou 310058, China}
\email{yangjiazhuo@zju.edu.cn}

\date{}


\begin{document}

\begin{abstract}
Let $1 \le k \le n$, let $f$ be a positive smooth function on the unit sphere, and consider the normalized flow
\[
\partial_t X = -f(\nu)\abs{X}^{\alpha}\sigma_k(\kappa)\nu + \beta X,
\qquad \beta = \binom{n}{k} .
\]
In the supercritical range $\alpha > k+1$ we prove time-independent $C^0$ and $C^1$ bounds, two-sided bounds for the speed factor $f(\nu)\abs{X}^{\alpha}\sigma_k(\kappa)$, and two-sided principal-curvature bounds. Consequently the flow exists for all time and converges exponentially in $C^{\infty}$, from every smooth strictly convex initial hypersurface enclosing the origin, to the unique smooth strictly convex solution of $f(x)r^{\alpha}\sigma_k(\kappa) = \beta u$. The convergence is driven by the relative residual $p = \partial_t\log u = \beta - f(x)r^{\alpha}\sigma_k(\kappa)/u$, whose $L^{\infty}$ norm is nonincreasing and decays with the explicit exponent $\beta(\alpha-k-1)$; the residual estimates are proved before, and independently of, the curvature estimates. The principal-radius estimate adapts the inverse-concavity argument of Li--Sheng--Wang: in inverse Gauss coordinates, the terms involving $\bn\log f$ and $\bn^2\log f$ are absorbed by the combined inverse-concavity--radial quadratic form. Existence of a strictly convex solution of the stationary equation in this range, for arbitrary positive angular data, was previously obtained by Bryan--Ivaki--Scheuer through an expanding-type flow started from a barrier; the contribution here is the convergence of the normalized contracting flow from arbitrary initial data, with an explicit exponential rate, together with uniqueness of the limit.
\end{abstract}

\maketitle

\section{Introduction}\label{sec:intro}

The $q$-dual Christoffel--Minkowski problem originates in the work of Huang, Lutwak, Yang and Zhang \cite{HLYZ16}, who introduced the dual curvature measures and posed their associated prescription problems. It is the counterpart in the dual Brunn--Minkowski theory of the classical Christoffel--Minkowski problem, in which one prescribes an intermediate curvature measure (or, equivalently, an elementary symmetric function of the principal radii); see \cite{GLL12, GLM09, GM03}. The parameter $q$ records the radial weight in the dual measure. We emphasize at the outset that this weight does not disappear at $q = 0$ at the level of the prescribed-density equation; see the discussion following \eqref{eq:pde}.

We now formulate the geometric measure and its equation. Let $K \subset \R^{n+1}$ be a convex body containing the origin in its interior, let $\rho_K$ be its radial function, and let $\alpha_K^{*}(\omega)$ denote the reverse radial Gauss image of a Borel set $\omega \subset \Sn$. We use the standard notions of convex geometry, for which \cite{Sch14} is the general reference. The $q$-th dual curvature measure of Huang, Lutwak, Yang and Zhang is
\begin{equation}\label{eq:measure}
\td{C}_q(K, \omega) = \frac{1}{n+1}\int_{\alpha_K^{*}(\omega)} \rho_K(\xi)^q \, d\xi .
\end{equation}
When $q > 0$, this measure has the weighted cone-volume representation
\begin{equation}\label{eq:cone}
\td{C}_q(K, \omega) = \frac{q}{n+1}\int_{C_K(\omega)} \abs{y}^{q-n-1} \, dy,
\end{equation}
where
\[
C_K(\omega) = \{ s\xi : \xi \in \alpha_K^{*}(\omega),\ 0 < s \le \rho_K(\xi) \}.
\]
Thus $\td{C}_q$ records a radially weighted volume of cones cut out by normal directions; for $q = n+1$ it is the usual cone-volume measure. This weighted cone-volume interpretation is the source of the power of the radial function appearing below.

For the rest of this paper we work in the smooth category. If $K$ is smooth and strictly convex, all measures involved are absolutely continuous with respect to the spherical Lebesgue measure $dx$, and their densities are computed by the change of variables between the radial parameter $\xi$ and the normal parameter $x$. Comparing the two expressions for the area element of $\partial K$,
\begin{equation}\label{eq:area-element}
dA = \frac{r^{n+1}}{u}\, d\xi = \sigma_n(\lambda)\, dx ,
\end{equation}
gives $d\xi = u\,r^{-n-1}\sigma_n(\lambda)\, dx$, and hence
\begin{equation}\label{eq:Cq-density}
\frac{d\td{C}_q(K, \cdot)}{dx} = \frac{1}{n+1}\, u\, r^{q-n-1}\sigma_n(\lambda),
\end{equation}
where $\lambda = (\lambda_1, \dots, \lambda_n)$ are the principal radii of $\partial K$, $u$ is the support function, and $r = \abs{X}$. For instance, the unit ball has $\td{C}_q(B_1, \Sn) = \abs{\Sn}/(n+1)$.

For $1 \le k \le n$ and smooth strictly convex $K$, Li, Sheng and Wang \cite{LSW20} considered the intermediate measure
\begin{equation}\label{eq:Ckq}
C_{k,q}(K, \omega) := \int_{\omega} \frac{1}{\sigma_{n-k}(\lambda_K(x))} \, d\td{C}_q(K, x) .
\end{equation}
Since the integrand in \eqref{eq:Ckq} involves the classical principal radii, we formulate $C_{k,q}$ directly for smooth strictly convex bodies; this is the only category needed in the present paper. Write $u$, $r$ and $\kappa = (\kappa_1, \dots, \kappa_n)$ for the support function, radial function and principal curvatures in inverse Gauss-map coordinates. Using \eqref{eq:Cq-density} and the identity
\begin{equation}\label{eq:sigma-quotient-intro}
\sigma_k(\kappa) = \frac{\sigma_{n-k}(\lambda)}{\sigma_n(\lambda)},
\end{equation}
we obtain
\begin{equation}\label{eq:Ckq-density}
\frac{dC_{k,q}(K, \cdot)}{dx}
= \frac{1}{n+1}\, u\, r^{q-n-1}\, \frac{\sigma_n(\lambda)}{\sigma_{n-k}(\lambda)}
= \frac{1}{n+1}\, \frac{u\, r^{q-n-1}}{\sigma_k(\kappa)} .
\end{equation}
Consequently, prescribing $C_{k,q}(K, \cdot) = \phi(x)\, dx$ is equivalent to the equation
\begin{equation}\label{eq:pde}
(n+1)\phi(x) r^{n+1-q}\sigma_k(\kappa) = u .
\end{equation}

When $k = n$, the factor in \eqref{eq:Ckq} is one, so that $C_{n,q} = \td{C}_q$ and \eqref{eq:pde} is the dual $q$-Minkowski equation. For $1 \le k < n$, we call the prescription of $C_{k,q}$ the $q$-dual Christoffel--Minkowski problem. Two caveats concerning this terminology are in order. First, even at $q = 0$ equation \eqref{eq:pde} reads $(n+1)\phi\, r^{n+1}\sigma_k(\kappa) = u$: writing $b[u] = \bn^2 u + u\bg$ for the curvature-radius tensor, so that $\sigma_k(\kappa) = \sigma_{n-k}(b)/\sigma_n(b)$, this is a Hessian-quotient equation with radial and support-function dependence, not an equation prescribing an elementary symmetric function of the principal radii alone. The vanishing of the radial weight in the integral \eqref{eq:measure} at $q = 0$ does not remove the reverse radial Gauss image or the Jacobian in \eqref{eq:area-element}. Thus the $q = 0$ prescription is not the classical Christoffel--Minkowski problem of \cite{GM03, GLM09}. Second, the equation also differs from the $L_p$ dual Christoffel--Minkowski problem of Ding--Li \cite{DL23}, whose smooth form is
\[
\sigma_j(b[u]) = \psi\, u^{p-1} r^{j+1-q} .
\]
For $1 \le k < n$, our curvature factor is the quotient $\sigma_{n-k}(b)/\sigma_n(b)$ rather than $\sigma_j(b)$, and the radial and support weights differ as well. The two prescriptions therefore involve different operators in the intermediate-curvature cases.

The natural parabolic equation associated with \eqref{eq:pde} is the contracting flow
\begin{equation}\label{eq:flow}
\partial_t X = -f(\nu)\abs{X}^{\alpha}\sigma_k(\kappa)\nu, \qquad \alpha = n+1-q,
\end{equation}
where $\nu$ is the outer unit normal and $f \in C^{\infty}(\Sn)$ is positive. Throughout this paper the elementary symmetric functions are unnormalized, so
\[
\sigma_k(\kappa) = \sum_{i_1 < \cdots < i_k} \kappa_{i_1}\cdots\kappa_{i_k},
\qquad
\beta := \sigma_k(1, \dots, 1) = \binom{n}{k}.
\]
We also adopt the convention $\sigma_0 = 1$, which covers the endpoint cases $k = 1$ and $k = n$ in the statements below.
After a time-dependent dilation, \eqref{eq:flow} becomes the normalized flow
\begin{equation}\label{eq:normflow}
\partial_t X = -f(\nu)r^{\alpha}\sigma_k(\kappa)\nu + \beta X .
\end{equation}
Its stationary equation in inverse Gauss-map coordinates is
\begin{equation}\label{eq:stationary}
f(x)r^{\alpha}\sigma_k(\kappa) = \beta u .
\end{equation}
Taking $f = \beta(n+1)\phi$ and $\alpha = n+1-q$ identifies \eqref{eq:stationary} with \eqref{eq:pde}. The supercritical condition
\begin{equation}\label{eq:supercritical}
\alpha > k+1
\end{equation}
is precisely $q < n-k$. In particular, for $k < n$ the positive-index range treated here is
\begin{equation}\label{eq:q-range}
0 < q < n-k .
\end{equation}

We next place the result in the development of these radial curvature flows. For $k = n$, Li, Sheng and Wang \cite{LSW20a} studied $f(\nu)r^{\alpha}K$ with arbitrary positive smooth angular factor $f$. They proved convergence for general data when $q < 0$, and in the positive range obtained the corresponding result under evenness and origin-symmetry; they also exhibited the obstruction to nonsymmetric convergence. The case $k = n$ of Theorem \ref{thm:main} below corresponds to $\alpha > n+1$, that is, $q < 0$, and thus falls within this general-data range. For general $1 \le k \le n$, Li, Sheng and Wang \cite{LSW20} then studied the speed $r^{\alpha}\sigma_k$. They proved smooth convergence after normalization for $\alpha \ge k+1$, but, because no analogue of the variational functional used for $k = n$ was available, their general-$k$ analysis was restricted to the isotropic case $f \equiv 1$.

The class of admissible initial hypersurfaces was enlarged by Li, Xu and Zhang \cite{LXZ22}. For the homogeneous speeds
\[
r^{\td{\alpha}/\td{\beta}}\, \sigma_k^{1/\td{\beta}},
\]
they extended the uniformly convex theory in the relevant parameter ranges to star-shaped $k$-convex hypersurfaces. More recently, Sheng and Yang \cite{SY27} replaced the power of $r$ by a nonhomogeneous radial profile and studied speeds of the form $\Phi(r)\sigma_k^{\theta}$ for star-shaped $k$-convex initial data. These results concern isotropic dependence on the normal direction. The case needed for prescribing a nonconstant density in \eqref{eq:pde} is instead the angularly anisotropic factor $f(\nu)$ in \eqref{eq:flow}. The present paper addresses this case for the unpowered curvature $\sigma_k$ in the supercritical range.

On the elliptic side, existence for \eqref{eq:pde} with arbitrary positive smooth density in the range \eqref{eq:q-range} is in fact already known. Bryan, Ivaki and Scheuer proved in \cite[Corollary 3.5]{BIS21} that for every positive $\psi \in C^{\infty}(\Sn)$ and $q + k < n$ there exists a strictly convex hypersurface satisfying
\begin{equation}\label{eq:bis}
\abs{X}^{n+1-q}\sigma_k(\kappa) = \langle X, \nu\rangle\, \psi(\nu) .
\end{equation}
Since $\langle X, \nu\rangle = u$, choosing $\psi = \big((n+1)\phi\big)^{-1}$ turns \eqref{eq:bis} into precisely \eqref{eq:pde}, and the constraint $q + k < n$ is precisely the supercritical condition \eqref{eq:supercritical}; for $q > 0$ this is the positive-index range \eqref{eq:q-range}. Their method is a curvature flow without a global term, of expanding (inverse) type, started from a round lower barrier inside an annulus, and convergence to a solution is proved for that particular evolution.

The contribution of the present paper is therefore dynamical rather than elliptic. We consider the normalized contracting flow \eqref{eq:normflow}, which is a different evolution: it is a contracting flow with the fixed dilational term $\beta X$, and we start it from an arbitrary smooth strictly convex hypersurface enclosing the origin rather than from a barrier. We prove that the solution exists for all time and converges exponentially in $C^{\infty}$ to a stationary solution, and that the stationary solution is unique in the supercritical range (Section \ref{sec:unique}). The analytical core is a principal-radius estimate for arbitrary positive angular data $f$ (Proposition \ref{prop:aniso-c2}), in which the terms involving $\bn f$ and $\bn^2 f$ are absorbed by a combined inverse-concavity--radial quadratic form, together with a residual argument that produces the exponential rate.

We also record the position of the uniqueness statement relative to existing elliptic criteria. Ding and Li proved in \cite[Theorem 6.1]{DL25} a general uniqueness theorem for star-shaped solutions of curvature equations $F(\kappa) = G(X, \nu)$ under the strict radial monotonicity condition $\partial_{\rho}(\rho G) < 0$, where $\rho = \abs{X}$ and the radial direction and the normal are held fixed. For the $k$-th root of the stationary equation \eqref{eq:stationary}, one has $F = \sigma_k^{1/k}$ and $G(X, \nu) = (\beta/f(\nu))^{1/k}\langle X, \nu\rangle^{1/k}\abs{X}^{-\alpha/k}$ on the positive-support domain, and $\partial_{\rho}(\rho G) = -\frac{\alpha-k-1}{k}G$: strict supercriticality is exactly their monotonicity condition (see Remark \ref{rem:dl-uniqueness}). Uniqueness in the supercritical range is therefore consistent with this general criterion, and we include a short support-function proof for completeness. The flow theorems of \cite{DL25} involve additional structural conditions on the data and a sign condition on the initial residual, so they do not yield the convergence statement of Theorem \ref{thm:main} for arbitrary positive $f$ and arbitrary admissible initial hypersurfaces.

\begin{theorem}\label{thm:main}
Let $1 \le k \le n$, $\alpha > k+1$ and $f \in C^{\infty}(\Sn)$ be positive. For every smooth, closed, strictly convex initial hypersurface enclosing the origin, the normalized flow \eqref{eq:normflow} has a unique smooth strictly convex solution $M_t$ for all $t \ge 0$.

Moreover, equation \eqref{eq:stationary} admits a unique smooth positive support function $u_{\infty}$ satisfying
\[
\bn^2 u_{\infty} + u_{\infty}\bg > 0
\]
and the hypersurfaces $M_t$ converge exponentially in $C^{\infty}$ to the stationary hypersurface determined by $u_{\infty}$. More precisely, with
\[
\omega := \beta(\alpha - k - 1) > 0,
\]
one has
\begin{equation}\label{eq:thm-rate}
\norm{u(\cdot, t) - u_{\infty}}_{C^{j}(\Sn)} \le C_j e^{-\omega t}
\qquad \text{for every } j \ge 0 \text{ and all } t \ge 0 .
\end{equation}
Each constant $C_j$ may depend on the initial hypersurface, $n$, $k$, $\alpha$, and $f$, but is independent of $t$.
\end{theorem}

Theorem \ref{thm:main} contains the following positive-index existence and uniqueness statement for the geometric measure problem.

\begin{corollary}[$q$-dual Christoffel--Minkowski problem]\label{cor:main}
Let $1 \le k < n$, $0 < q < n-k$, and let $\phi \in C^{\infty}(\Sn)$ be positive. Then there exists a unique smooth positive support function $u$ such that
\[
\bn^2 u + u\bg > 0
\]
and
\begin{equation}\label{eq:cor-pde}
(n+1)\phi(x) r^{n+1-q}\sigma_k(\kappa) = u .
\end{equation}
Equivalently, every smooth positive density is the density of $C_{k,q}$ for a unique smooth strictly convex body containing the origin.
\end{corollary}

For convergence, the essential difficulty is exactly the one identified in \cite{LSW20}: for $k < n$ there is no known analogue of the functional used in the Gauss curvature case. We replace it by a pointwise residual. Namely, let
\begin{equation}\label{eq:residual-intro}
p = \frac{u_t}{u} = \partial_t \log u = \beta - \frac{f r^{\alpha}\sigma_k(\kappa)}{u}.
\end{equation}
We prove directly that
\begin{equation}\label{eq:mono-intro}
t \longmapsto \norm{p(\cdot, t)}_{L^{\infty}(\Sn)}
\end{equation}
is nonincreasing whenever $\alpha \ge k+1$. In fact, parabolic comparison with spatial constants gives two-sided time-independent bounds for $f r^{\alpha}\sigma_k/u$ determined by the initial datum alone, so \eqref{eq:mono-intro} decays exponentially whenever $\alpha > k+1$, with the exponent $\beta(\alpha-k-1)$ (see Remark \ref{rem:sharper-rate}). These residual estimates use only the smoothness and strict convexity of the solution together with the time-independence of $f$ in inverse Gauss-map coordinates; they are derived in Section \ref{subsec:residual}, before the curvature estimates. The residual estimate applies throughout every smooth strictly convex existence interval, and combined with the uniform $C^0$ bounds established in Section \ref{sec:c0c1} it yields convergence of the whole flow.

The paper is organized as follows. Section \ref{sec:prelim} records the radial and inverse Gauss-map equations and the ellipticity convention. Section \ref{sec:c0c1} proves the $C^0$ and $C^1$ estimates. In Section \ref{subsec:residual} we derive the evolution of the residual $p$ and obtain the two-sided speed and $\sigma_k$ bounds by comparison on compact time intervals; Section \ref{subsec:c2} proves the $C^2$ estimate. Section \ref{sec:reg} proves higher regularity and long-time existence. Section \ref{sec:conv} contains the exponential convergence argument. Section \ref{sec:unique} proves uniqueness of the stationary solution, discusses its relation to the general criterion of \cite{DL25}, and closes with a remark on the critical exponent $\alpha = k+1$.

\section{Geometric preliminaries and scalar equations}\label{sec:prelim}

Throughout this paper, we always use $\bg$ and $\bn$ to denote the standard metric and covariant derivative on $\Sn$. We also fix the terminology: \emph{strictly convex} means that the second fundamental form is positive definite; equivalently, all principal curvatures are positive.

\subsection{Radial coordinates}
Let $M \subset \R^{n+1}$ be a smooth, closed hypersurface which is star-shaped with respect to the origin. Thus $M$ is represented as the radial graph
\begin{equation}\label{eq:radialgraph}
X(\xi) = r(\xi)\xi, \qquad \xi \in \Sn,
\end{equation}
where $r : \Sn \to \R_{+}$ is the radial function. The following standard radial-graph identities are recorded, in the same static form, in \cite[Section 2]{LSW20}. The induced metric and its inverse are
\begin{equation}\label{eq:metric}
g_{ij} = r^2 \bg_{ij} + r_i r_j,
\qquad
g^{ij} = r^{-2}\Big( \bg^{ij} - \frac{r_i r_j}{r^2 + \abs{\bn r}^2} \Big).
\end{equation}
The outer unit normal and support function are therefore
\begin{equation}\label{eq:normal}
\nu = \frac{r\xi - \bn r}{\sqrt{r^2 + \abs{\bn r}^2}} = \frac{\xi - \bn\vartheta}{\sqrt{1 + \abs{\bn\vartheta}^2}},
\qquad \vartheta = \log r,
\end{equation}
and
\begin{equation}\label{eq:support-radial}
u = \langle X, \nu\rangle = \frac{r^2}{\sqrt{r^2 + \abs{\bn r}^2}} = \frac{r}{\sqrt{1 + \abs{\bn\log r}^2}} .
\end{equation}
With the convention that $h_{ij} = \langle \vec{h}_{ij}, -\nu\rangle$, one has
\begin{equation}\label{eq:hij}
h_{ij} = \frac{r^2\bg_{ij} + 2 r_i r_j - r r_{ij}}{\sqrt{r^2 + \abs{\bn r}^2}} .
\end{equation}
The principal curvatures of $M$ are the eigenvalues of $h_{ij}$ with respect to $g_{ij}$. Equivalently, they are the eigenvalues of the symmetric Weingarten matrix $a = (g^{-1/2})h(g^{-1/2})$. In terms of $\vartheta = \log r$, this matrix can be written as
\begin{equation}\label{eq:weingarten-matrix}
a_{ij} = e^{-\vartheta}\big(1 + \abs{\bn\vartheta}^2\big)^{-1/2} \gamma^{il}\big( \delta_{lm} + \vartheta_l\vartheta_m - \vartheta_{lm} \big)\gamma^{mj},
\end{equation}
where
\[
\gamma^{ij} = \delta_{ij} - \frac{\vartheta_i\vartheta_j}{\sqrt{1 + \abs{\bn\vartheta}^2}\,\big(1 + \sqrt{1 + \abs{\bn\vartheta}^2}\,\big)} .
\]

\subsection{Inverse Gauss map coordinates}
We follow the notation and presentation of \cite[Section 2]{LSW20a}. Let $M \subset \R^{n+1}$ be a smooth, closed, strictly convex hypersurface, parametrized by its inverse Gauss map $X : \Sn \to M$. The support function $u : \Sn \to \R$ of $M$ is defined by
\begin{equation}\label{eq:support-def}
u(x) = \sup\{ \langle x, y\rangle : y \in M \}.
\end{equation}
The supremum is attained at $y = X(x)$, where $x$ is the outer unit normal of $M$ at $y$. It follows that
\begin{equation}\label{eq:position}
X(x) = u(x)x + \bn u(x),
\end{equation}
and hence the radial function is
\begin{equation}\label{eq:r-from-u}
r = \abs{X} = \sqrt{u^2 + \abs{\bn u}^2} .
\end{equation}
Let $W = d\nu : T_{X(x)}M \to T_x\Sn$ be the Weingarten map, where both tangent spaces are identified with $x^{\perp} \subset \R^{n+1}$. Since $\nu \circ X = \mathrm{id}_{\Sn}$, we have $dX_x = W^{-1}$. We lower an index of the inverse Weingarten map with the spherical metric $\bg$ and define the curvature-radius tensor
\begin{equation}\label{eq:radius-tensor}
b_{ij} := \bg_{il}(W^{-1})^{l}{}_{j} = u_{ij} + u\bg_{ij}.
\end{equation}
Indeed, the second equality follows by differentiating \eqref{eq:position}. We arrange the eigenvalues of $b$ with respect to $\bg$ in nondecreasing order and the principal curvatures in nonincreasing order,
\[
\lambda_1 \le \cdots \le \lambda_n,
\qquad
\kappa_1 \ge \cdots \ge \kappa_n .
\]
Since the eigenvalues of $b$ are the principal radii, these ordering conventions give $\lambda_i = \kappa_i^{-1}$ for each $i$, and hence
\begin{equation}\label{eq:G-def}
\sigma_k(\kappa) = \frac{\sigma_{n-k}(\lambda(b))}{\sigma_n(\lambda(b))} =: G(b).
\end{equation}
Then, at a frame diagonalizing $b$,
\begin{equation}\label{eq:Gdot-negative}
\dot{G}^{ii} = -\kappa_i^2\, \sigma_{k-1}(\kappa \,|\, i) < 0 .
\end{equation}
Thus $-\dot{G}^{ij}$ is positive definite on the positive cone.

\subsection{Ellipticity and inverse concavity}
Let
\[
F(h) = \sigma_k(h)^{1/k}.
\]
We use $F^{ij}$ and $F^{ij,rs}$ for its first and second derivatives. The following inverse-concavity inequality is standard; see, for example, \cite{And94, LSW20}; the general properties of the elementary symmetric functions of the principal curvatures used throughout are developed in \cite{CNS85}.

\begin{lemma}\label{lem:invconc}
If $h > 0$, $\td{h} = h^{-1}$, and $\eta$ is symmetric, then
\begin{equation}\label{eq:invconc}
\big( F^{ij,rs} + 2 F^{ir}\td{h}^{js} \big)\eta_{ij}\eta_{rs} \ge 2F^{-1}\big( F^{ij}\eta_{ij} \big)^2 .
\end{equation}
\end{lemma}

\noindent Equivalently, with $G$ as in \eqref{eq:G-def}, the function
\begin{equation}\label{eq:tildeF}
\td{F}(b) := G(b)^{-1/k} = \Big( \frac{\sigma_n(b)}{\sigma_{n-k}(b)} \Big)^{1/k}
\end{equation}
is concave on the positive cone, and $F(h) = \td{F}(h^{-1})^{-1}$.

\subsection{Normalization of the flow}
We record the rescaling that transforms the unnormalized flow \eqref{eq:flow} into \eqref{eq:normflow}. Write $s$ for the unnormalized time and set
\begin{equation}\label{eq:rescaling}
X(\cdot, s) = \lambda(s)\, \td{X}(\cdot, \tau(s)),
\qquad
\lambda(0) = 1,\quad \tau(0) = 0.
\end{equation}
A dilation does not change the unit normal, while the radial function and the principal curvatures scale according to
\begin{equation}\label{eq:scaling}
r[X] = \lambda \td{r},
\qquad
\kappa_i[X] = \lambda^{-1}\td{\kappa}_i,
\qquad
\sigma_k(\kappa[X]) = \lambda^{-k}\sigma_k(\td{\kappa}).
\end{equation}
Differentiating \eqref{eq:rescaling} with respect to $s$ and using \eqref{eq:flow} and \eqref{eq:scaling}, we obtain
\[
\lambda' \td{X} + \lambda \tau' \partial_{\tau}\td{X} = -\lambda^{\alpha - k} f(\td{\nu})\, \td{r}^{\,\alpha}\sigma_k(\td{\kappa})\, \td{\nu}.
\]
Hence
\begin{equation}\label{eq:tauflow}
\partial_{\tau}\td{X} = -\frac{\lambda^{\alpha-k-1}}{\tau'}\, f(\td{\nu})\, \td{r}^{\,\alpha}\sigma_k(\td{\kappa})\, \td{\nu} - \frac{\lambda'}{\lambda\tau'}\, \td{X}.
\end{equation}
Choose the new time and the dilation factor by
\begin{equation}\label{eq:timechoice}
\tau' = \lambda^{\alpha-k-1},
\qquad
\lambda' = -\beta \lambda^{\alpha-k}.
\end{equation}
Then \eqref{eq:tauflow} becomes
\[
\partial_{\tau}\td{X} = -f(\td{\nu})\, \td{r}^{\,\alpha}\sigma_k(\td{\kappa})\, \td{\nu} + \beta \td{X},
\]
which is precisely the normalized flow \eqref{eq:normflow}.

For completeness, put $m = \alpha - k - 1 > 0$. Solving \eqref{eq:timechoice} gives
\begin{equation}\label{eq:lambda-tau}
\lambda(s) = (1 + m\beta s)^{-1/m},
\qquad
\tau(s) = \frac{1}{m\beta}\log(1 + m\beta s).
\end{equation}
Equivalently,
\[
s(\tau) = \frac{e^{m\beta\tau} - 1}{m\beta},
\qquad
\td{X}(\cdot, \tau) = e^{\beta\tau} X(\cdot, s(\tau)).
\]
In particular, infinite unnormalized time corresponds to infinite normalized time. From now on we drop the tildes and again write $t$ for the normalized time. The term $\beta X$ has normal component $\beta u\nu$; its tangential component only changes the parametrization. Combining the normal velocity with the static identities in the preceding subsections gives the radial equation
\begin{equation}\label{eq:radial-eq}
r_t = -\sqrt{1 + \abs{\bn\log r}^2}\, f(\nu) r^{\alpha}\sigma_k(\kappa) + \beta r
\end{equation}
and, in inverse Gauss-map coordinates, the support-function equation
\begin{equation}\label{eq:support-eq}
u_t = -f(x)r^{\alpha}G(b) + \beta u =: -S + \beta u .
\end{equation}
Here $f = f(x)$ is independent of time in inverse Gauss-map coordinates. This observation will be important in the residual and speed estimates. In this paper we take
\[
\beta = c_k = \sigma_k(1, \dots, 1) = \binom{n}{k}.
\]

\section{Uniform bounds for the radial and support functions}\label{sec:c0c1}

Put
\begin{equation}\label{eq:data}
c_k = \binom{n}{k},
\qquad
m = \alpha - k - 1 > 0,
\qquad
f_{-} = \min_{\Sn} f,
\qquad
f_{+} = \max_{\Sn} f .
\end{equation}

\begin{lemma}\label{lem:r-bound}
There are positive constants $r_{-}$ and $r_{+}$, depending only on the initial radial function and the fixed data in \eqref{eq:data}, such that
\begin{equation}\label{eq:r-bound}
r_{-} \le r(\xi, t) \le r_{+}
\end{equation}
for every $\xi \in \Sn$ and every time for which a smooth strictly convex solution exists.
\end{lemma}

\begin{proof}
At a spatial minimum of $r$, one has $\bn r = 0$, $\bn^2 r \ge 0$, and hence
\[
\sigma_k(\kappa) \le c_k r^{-k}.
\]
Using \eqref{eq:radial-eq},
\begin{equation}\label{eq:rmin-ode}
\frac{d}{dt} r_{\min} \ge r_{\min}\big( \beta - c_k f_{+} r_{\min}^{m} \big).
\end{equation}
Similarly, at a spatial maximum,
\begin{equation}\label{eq:rmax-ode}
\frac{d}{dt} r_{\max} \le r_{\max}\big( \beta - c_k f_{-} r_{\max}^{m} \big).
\end{equation}
The differential inequalities for $r_{\min}$ and $r_{\max}$ are understood almost everywhere, or equivalently in the usual barrier sense. Since $m > 0$, comparison with the scalar ODEs gives
\begin{equation}\label{eq:rmin}
r_{\min}(t) \ge \min\Big\{ r_{\min}(0), \Big( \frac{\beta}{c_k f_{+}} \Big)^{1/m} \Big\},
\end{equation}
\begin{equation}\label{eq:rmax}
r_{\max}(t) \le \max\Big\{ r_{\max}(0), \Big( \frac{\beta}{c_k f_{-}} \Big)^{1/m} \Big\}.
\end{equation}
\end{proof}

\begin{lemma}\label{lem:u-bound}
Under the hypotheses of Lemma \ref{lem:r-bound},
\begin{equation}\label{eq:u-bound}
C^{-1} \le u \le C,
\qquad
\abs{\bn u} + \abs{\bn r} \le C,
\qquad
\langle \xi, \nu\rangle \ge C^{-1}.
\end{equation}
\end{lemma}

\begin{proof}
For every convex body containing the origin,
\begin{equation}\label{eq:minmax-ur}
\min_{\Sn} u = \min_{\Sn} r,
\qquad
\max_{\Sn} u = \max_{\Sn} r .
\end{equation}
The identity \eqref{eq:position} gives
\[
\abs{\bn u} \le r_{+}.
\]
Moreover,
\[
\frac{u}{r} = \langle \xi, \nu\rangle = \frac{r}{\sqrt{r^2 + \abs{\bn r}^2}},
\]
and therefore
\[
\abs{\bn r} \le \frac{r_{+}^2}{r_{-}},
\qquad
\langle \xi, \nu\rangle \ge \frac{r_{-}}{r_{+}} .
\]
\end{proof}

\begin{remark}\label{rem:angle}
The gradients in \eqref{eq:u-bound} live in different spherical coordinates: $\bn r$ is taken with respect to the radial direction $\xi$, while $\bn u$ is taken with respect to the normal $x$. The angle estimate $\langle\xi, \nu\rangle \ge C^{-1}$ says that the radial parametrization remains uniformly transverse: the hypersurface stays a uniformly star-shaped radial graph, and the support function controls the radial function and its gradient. These first-derivative bounds use no curvature estimate. We note for precision that they do not by themselves control the differential of the radial Gauss map: in inverse Gauss coordinates $\xi(x) = X(x)/r(x)$, and differentiation gives
\[
D\xi = \frac{1}{r}\big( I - \xi \otimes \xi \big) b ,
\]
whose norm depends on the principal radii. Those eigenvalues are estimated only in Section \ref{subsec:c2}.
\end{remark}

\section{Curvature estimates}\label{sec:curvature}

We first derive the evolution equation of the relative residual and deduce the two-sided bounds for the contracting speed factor
\[
S = f(x) r^{\alpha}\sigma_k(\kappa)
\]
by parabolic comparison, independently of any curvature estimate. We then prove the principal-radius estimate. Throughout this section the constants may depend on the bounds in Lemmas \ref{lem:r-bound} and \ref{lem:u-bound}.

\subsection{The relative residual and the two-sided speed bound}\label{subsec:residual}

Recall the support-function equation \eqref{eq:support-eq} in inverse Gauss-map coordinates, with $G(b)$ defined in \eqref{eq:G-def}, and define the parabolic operator on $\Sn \times [0, T)$ by
\begin{equation}\label{eq:Lop}
\cL := \partial_t + f r^{\alpha} \dot{G}^{ij}\, \bn_i\bn_j ,
\end{equation}
where $\dot{G}^{ij} = \partial G/\partial b_{ij}$, so that the matrix $\big( \dot{G}^{ij} \big)$ is negative definite by \eqref{eq:Gdot-negative}, and $[0,T)$ is any interval on which the solution remains smooth and strictly convex.

\begin{lemma}\label{lem:residual}
The relative residual
\[
p := \partial_t \log u = \frac{u_t}{u} = \beta - \frac{S}{u}
\]
satisfies
\begin{equation}\label{eq:residual-eq}
\cL p = -\Big( \alpha f r^{\alpha-2}G\, u_i + 2 f r^{\alpha} \dot{G}^{ij}\frac{u_j}{u} \Big) p_i - (\alpha - k - 1)\frac{f r^{\alpha}G}{u}\, p .
\end{equation}
\end{lemma}

\begin{proof}
Denote $w = u_t$. Since
\begin{equation}\label{eq:rt-bt}
r_t = \frac{uw + u_i w_i}{r},
\qquad
(b_{ij})_t = w_{ij} + w \bg_{ij},
\end{equation}
differentiating \eqref{eq:support-eq} in time, and using that $f = f(x)$ is independent of time in inverse Gauss-map coordinates, gives
\begin{align}
w_t &= -\alpha f r^{\alpha-2}(uw + u_i w_i)G - f r^{\alpha}\dot{G}^{ij}(b_{ij})_t + \beta w \notag \\
&= -f r^{\alpha}\dot{G}^{ij}w_{ij} - \alpha f r^{\alpha-2}G u_i w_i - f r^{\alpha}\dot{G}^{ij}\bg_{ij} w - \alpha f r^{\alpha-2}u G w + \beta w . \label{eq:wt}
\end{align}
It follows from \eqref{eq:wt} that
\begin{equation}\label{eq:Lw}
\cL w = -\alpha f r^{\alpha-2}G u_i w_i - f r^{\alpha}\dot{G}^{ij}\bg_{ij} w - \alpha f r^{\alpha-2}u G w + \beta w .
\end{equation}
Using \eqref{eq:support-eq}, the identity $\bn_i\bn_j u = b_{ij} - u\bg_{ij}$, and the $(-k)$-homogeneity of $G$, we have
\begin{equation}\label{eq:Lu}
\cL u = -(k+1) f r^{\alpha}G - f r^{\alpha} u\, \dot{G}^{ij}\bg_{ij} + \beta u .
\end{equation}
Using \eqref{eq:Lw} and \eqref{eq:Lu}, we have
\begin{align}
\cL p &= \frac{\cL w}{u} - \frac{p\, \cL u}{u} - 2 f r^{\alpha}\dot{G}^{ij}\, \bn_i\log u\, \bn_j p \notag \\
&= -\alpha f r^{\alpha-2}G \frac{u_i}{u} w_i - \alpha f r^{\alpha-2}G u p \notag \\
&\qquad + (k+1)\frac{f r^{\alpha}G}{u} p \notag \\
&\qquad - 2 f r^{\alpha}\dot{G}^{ij}\frac{u_j}{u}\, \bn_i p . \label{eq:LQ-expand}
\end{align}
Since $w_i = u_i p + u p_i$, the first two terms of \eqref{eq:LQ-expand} combine according to
\begin{align*}
-\alpha f r^{\alpha-2}G \frac{u_i}{u}\big( u_i p + u p_i \big) &- \alpha f r^{\alpha-2}G u p \\
&= -\alpha f r^{\alpha-2}G u_i p_i - \alpha f r^{\alpha-2}G\, \frac{\abs{\bn u}^2 + u^2}{u}\, p \\
&= -\alpha f r^{\alpha-2}G u_i p_i - \alpha\, \frac{f r^{\alpha}G}{u}\, p ,
\end{align*}
using $r^2 = u^2 + \abs{\bn u}^2$ in the last step. It follows that
\begin{equation*}
\cL p = -\Big( \alpha f r^{\alpha-2}G\, u_i + 2 f r^{\alpha}\dot{G}^{ij}\frac{u_j}{u} \Big) p_i - (\alpha - k - 1)\frac{f r^{\alpha}G}{u}\, p .
\qedhere
\end{equation*}
\end{proof}

The following comparison principle is standard; we record it here for the reader's convenience.

\begin{lemma}[Comparison principle {\cite[Chapter II, Corollary 2.5]{Lie96}}]\label{lem:comparison}
Let $T > 0$, and consider the parabolic operator
\[
\cC[\phi] = \phi_t - a^{ij}\bn_i\bn_j\phi + d^i\bn_i\phi + c\phi
\qquad \text{for } \phi \in C^{\infty}(\Sn \times [0, T]),
\]
where $a^{ij}$ is a positive definite contravariant $2$-tensor field, $d = d^i\partial_i$ is a vector field, and $c \ge 0$ is a smooth function. If $\phi, \psi \in C^{\infty}(\Sn \times [0, T])$ satisfy
\[
\cC[\phi] \le \cC[\psi],
\qquad
\phi(\cdot, 0) \le \psi(\cdot, 0),
\]
then $\phi \le \psi$ on $\Sn \times [0, T]$.
\end{lemma}

We now derive the two-sided speed bound, the $L^{\infty}$ monotonicity, and the exponential estimate directly from \eqref{eq:residual-eq} and the comparison principle. The argument uses only the smoothness and strict convexity of the solution; in particular, no curvature estimate enters.

\begin{proposition}\label{prop:residual-bounds}
Let $\alpha \ge k+1$ and let $[0, T]$ be any compact subinterval of the smooth, strictly convex existence interval. Then:
\begin{enumerate}
\item[(i)] For every $t \in [0, T]$,
\begin{equation}\label{eq:p-bound}
p_0^{-} \le p(\cdot, t) \le p_0^{+},
\end{equation}
where $p_0^{+} := \max\big\{ 0, \max_{\Sn} p(\cdot, 0) \big\}$ and $p_0^{-} := \min\big\{ 0, \min_{\Sn} p(\cdot, 0) \big\}$.
In particular, $t \mapsto E(t) := \norm{p(\cdot, t)}_{L^{\infty}(\Sn)}$ is nonincreasing on the existence interval.
\item[(ii)] With
\begin{equation}\label{eq:hpm}
h_{-} := \min\Big\{ \beta, \min_{\Sn} \frac{S(\cdot, 0)}{u(\cdot, 0)} \Big\},
\qquad
h_{+} := \max\Big\{ \beta, \max_{\Sn} \frac{S(\cdot, 0)}{u(\cdot, 0)} \Big\},
\end{equation}
one has
\begin{equation}\label{eq:Su-bound}
0 < h_{-} \le \frac{S(\cdot, t)}{u(\cdot, t)} \le h_{+}
\qquad \text{for all } t \in [0, T].
\end{equation}
\item[(iii)] If $\alpha > k+1$, then
\begin{equation}\label{eq:c0}
c(x, t) := (\alpha - k - 1)\frac{S}{u} \ge c_0 := (\alpha - k - 1) h_{-} > 0,
\end{equation}
and
\begin{equation}\label{eq:exp-decay}
E(t) \le e^{-c_0 t} E(0)
\qquad \text{for all } t \in [0, T].
\end{equation}
\end{enumerate}
\end{proposition}

\begin{proof}
Put
\begin{gather*}
a^{ij} = -f r^{\alpha}\dot{G}^{ij} > 0,
\qquad
c(x, t) = (\alpha - k - 1)\frac{S}{u} \ge 0, \\
d^i = \alpha f r^{\alpha-2}G\, u_i + 2 f r^{\alpha}\dot{G}^{ij}\frac{u_j}{u},
\end{gather*}
where the sign of $c$ uses $\alpha \ge k+1$ and $S > 0$, and introduce the parabolic operator
\begin{equation}\label{eq:Cop}
\cC[\phi] = \phi_t - a^{ij}\phi_{ij} + d^i\phi_i + c\phi .
\end{equation}
Equation \eqref{eq:residual-eq} is precisely
\begin{equation}\label{eq:Cp}
\cC[p] = 0 .
\end{equation}

(i) Since $p_0^{+} \ge 0$, the constant function $p_0^{+}$ satisfies
\[
\cC[p_0^{+}] = c\, p_0^{+} \ge 0 = \cC[p],
\qquad
p_0^{+} \ge p(\cdot, 0),
\]
so Lemma \ref{lem:comparison} gives $p \le p_0^{+}$ on $\Sn \times [0, T]$. Similarly, $p_0^{-} \le 0$ gives $\cC[p_0^{-}] = c\, p_0^{-} \le 0 = \cC[p]$, hence $p \ge p_0^{-}$. Applying the same comparison on $[t_0, T]$ with the constants $\pm E(t_0)$, for any $t_0 \in [0, T)$, gives $\abs{p(\cdot, t)} \le E(t_0)$ for $t \ge t_0$; hence $E$ is nonincreasing.

(ii) Since $p = \beta - S/u$, the bound $p \le p_0^{+}$ gives
\[
\frac{S}{u} \ge \beta - p_0^{+} = \min\Big\{ \beta, \min_{\Sn}\frac{S(\cdot, 0)}{u(\cdot, 0)} \Big\} = h_{-},
\]
and $h_{-} > 0$ because $u(\cdot, 0) > 0$ and $S(\cdot, 0) > 0$ by the strict convexity of the initial hypersurface. Similarly, $p \ge p_0^{-}$ gives $S/u \le \beta - p_0^{-} = h_{+}$.

(iii) If $\alpha > k+1$, then \eqref{eq:Su-bound} gives $c(x, t) \ge (\alpha - k - 1)h_{-} = c_0 > 0$. Set $M_0 = E(0)$ and define
\[
\Phi(t) = M_0 e^{-c_0 t}.
\]
Because $\Phi$ is spatially constant,
\[
\cC[\Phi] = (c - c_0)\Phi \ge 0,
\qquad
\cC[-\Phi] = -(c - c_0)\Phi \le 0 .
\]
At the initial slice $-\Phi(0) \le p(\cdot, 0) \le \Phi(0)$. Applying Lemma \ref{lem:comparison} once more gives
\[
-\Phi(t) \le p(\cdot, t) \le \Phi(t),
\qquad t \in [0, T],
\]
which is \eqref{eq:exp-decay}.
\end{proof}

\begin{remark}[A sharper exponential rate]\label{rem:sharper-rate}
The exponential rate in \eqref{eq:exp-decay} can in fact be taken independent of the initial datum. Put
\[
m = \alpha - k - 1 > 0,
\qquad
H = \frac{S}{u} = \beta - p .
\]
By \eqref{eq:Cp}, the function $H$ satisfies
\begin{equation}\label{eq:H-eq}
H_t - a^{ij}H_{ij} + d^iH_i = mH(\beta - H) .
\end{equation}
For a spatial constant $h$, equation \eqref{eq:H-eq} reduces to the logistic equation $h' = mh(\beta - h)$, whose solution with $h(0) > 0$ is
\begin{equation}\label{eq:logistic}
h(t) = \frac{\beta}{1 + \big( \beta/h(0) - 1 \big) e^{-m\beta t}} .
\end{equation}
Let $h^{-}$ and $h^{+}$ be the solutions \eqref{eq:logistic} with $h^{-}(0) = \min_{\Sn} H(\cdot, 0)$ and $h^{+}(0) = \max_{\Sn} H(\cdot, 0)$. The difference $\phi = H - h^{\pm}$ satisfies
\[
\phi_t - a^{ij}\phi_{ij} + d^i\phi_i + m\big( H + h^{\pm} - \beta \big)\phi = 0 .
\]
The zero-order coefficient need not be nonnegative, but the exponential change of unknown $\psi = e^{-m\beta t}\phi$ turns it into $m\big( H + h^{\pm} \big) > 0$, so Lemma \ref{lem:comparison}, applied to $\psi$, gives
\[
h^{-}(t) \le H(\cdot, t) \le h^{+}(t) .
\]
Since $\abs{\beta - h(t)} = \beta\, \abs{\beta/h(0) - 1}\, e^{-m\beta t} / \big( 1 + (\beta/h(0) - 1) e^{-m\beta t} \big)$, one obtains
\begin{equation}\label{eq:sharper-rate}
E(t) = \norm{p(\cdot, t)}_{L^{\infty}(\Sn)} \le \Big( 1 + \frac{\beta}{h^{-}(0)} \Big) E(0)\, e^{-m\beta t},
\qquad
m\beta = \beta(\alpha - k - 1) .
\end{equation}
Thus the exponent is independent of the initial datum; only the prefactor depends on it. The Schauder argument of Section \ref{sec:conv} then yields the same exponent in every $C^j$ norm, $j \ge 0$.
\end{remark}

\begin{corollary}[Two-sided speed and $\sigma_k$ bounds]\label{cor:speed}
There exists $C \ge 1$, depending only on the initial datum and the fixed data, such that
\begin{equation}\label{eq:speed-bounds}
C^{-1} \le S = f(x)r^{\alpha}\sigma_k(\kappa) \le C,
\qquad
C^{-1} \le \sigma_k(\kappa) \le C
\end{equation}
on the whole smooth strictly convex existence interval.
\end{corollary}

\begin{proof}
Proposition \ref{prop:residual-bounds}(ii) and Lemma \ref{lem:u-bound} give
\[
h_{-} r_{-} \le S \le h_{+} r_{+} .
\]
Dividing by $f r^{\alpha}$ and using the fixed two-sided bounds for $f$ and $r$ proves the asserted bounds for $\sigma_k$. Notice that neither $\bn f$ nor $\bn^2 f$ occurs.
\end{proof}

\subsection{The principal-radius estimate for general data}\label{subsec:c2}

The missing lower bound for the principal curvatures is equivalent to an upper bound for the largest eigenvalue of the radius tensor $b_{ij}$. Among the estimates up to second order, the principal-radius estimate is the first to involve spatial derivatives of $f$. We give the calculation in detail: first replace the largest eigenvalue by a smooth touching component, and then keep the inverse-concavity and radial terms in a single quadratic form before estimating the data-dependent terms. The isotropic case follows as a consequence (Corollary \ref{cor:iso-c2}); for the corresponding intrinsic calculation in Euclidean space, see \cite[Lemma 3.7]{LSW20}.

\begin{proposition}[The anisotropic $C^2$ estimate]\label{prop:aniso-c2}
Let $f \in C^2(\Sn)$ be positive and let $\alpha \ge k+1$. Suppose a smooth strictly convex solution has uniform positive upper and lower bounds for $r$ and $u$. Then
\begin{equation}\label{eq:aniso-c2}
\lambda_{\max}(b_{ij}) \le C,
\end{equation}
where $C$ depends only on $n$, $k$, $\alpha$, the assumed two-sided bounds for $r$ and $u$, $f_{-}^{-1}$, $\norm{f}_{C^2(\Sn)}$, and the initial maximum of $\Lambda/u$.
\end{proposition}

\begin{proof}
We use inverse Gauss-map coordinates, so the background metric $\bg$ is fixed and $f = f(x)$ is independent of time. Let $\Lambda(x, t)$ be the largest eigenvalue of $b_{ij}$ and consider
\begin{equation}\label{eq:W-def}
W = \log\Lambda - \log u .
\end{equation}
Fix $0 < T' < T$ and suppose that $W$ attains its maximum on $\Sn \times [0, T']$ at $(x_0, t_0)$ with $t_0 > 0$; otherwise the estimate follows from the initial data. Choose a $\bg$-unit eigenvector $e_1$ of $b$ at the maximum, extend $e_1$ to a vector field that remains $\bg$-unit in a spatial neighborhood of $x_0$, is independent of $t$, and satisfies $\bn e_1 = 0$ at $x_0$ (for instance, extend it as a constant field in a normal coordinate chart centered at $x_0$ and normalize), and put
\[
\lambda(x, t) = b(e_1, e_1),
\qquad
w = \log\lambda - \log u .
\]
The Rayleigh characterization gives $\lambda \le \Lambda$, with equality at $(x_0, t_0)$. Thus $w$ is a smooth lower barrier for $W$ which also attains a local maximum there. In a frame diagonalizing $b$ at the touching point, $\lambda = b_{11}$; differentiating $\abs{e_1}^2 = 1$ twice at $x_0$ and using $\bn e_1 = 0$ gives $\langle \bn_i\bn_j e_1, e_1\rangle = 0$, so the extension contributes no additional second-derivative terms to $b_{11;ij}$, and
\begin{align}
\frac{b_{11;i}}{b_{11}} &= \frac{u_i}{u}, \label{eq:first-deriv} \\
0 \ge w_{ij} &= \frac{b_{11;ij}}{b_{11}} - \frac{b_{11;i}b_{11;j}}{b_{11}^2} - \frac{u_{ij}}{u} + \frac{u_i u_j}{u^2}. \label{eq:w-hess}
\end{align}

\medskip\noindent\emph{Step 1: the tensor evolution and the normalized operator.}
From $b_{ij} = u_{ij} + u\bg_{ij}$ and \eqref{eq:support-eq},
\begin{equation}\label{eq:b-evol}
(b_{ij})_t = -S_{ij} - S\bg_{ij} + \beta b_{ij}.
\end{equation}
Write
\[
\gamma = \log f,
\qquad
P = r^{\alpha}G(b),
\qquad
S = e^{\gamma}P,
\]
and define
\begin{equation}\label{eq:ahat}
a^{ij} := -\frac{\partial S}{\partial b_{ij}} = -f r^{\alpha}\dot{G}^{ij} > 0,
\qquad
\wh{a}^{ij} := S^{-1}a^{ij} = -\frac{\dot{G}^{ij}}{G},
\end{equation}
and the normalized operator
\[
\cP := S^{-1}\partial_t - \wh{a}^{ij}\bn_i\bn_j .
\]
The normalization by $S$ removes the size of the speed from the principal part. We shall also use the sphere commutation and homogeneity identities
\begin{align}
b_{pq;11} &= b_{11;pq} + \bg_{11}b_{pq} - \bg_{pq}b_{11}, \label{eq:commutation} \\
\dot{G}^{pq}b_{pq} &= -kG,
\qquad
\wh{a}^{pq}b_{pq} = k . \label{eq:homog}
\end{align}
At the maximum of the barrier, \eqref{eq:w-hess} and ellipticity imply $\cP w \ge 0$.

\medskip\noindent\emph{Step 2: the exact normalized operator.}
At the touching point write
\[
\lambda = b_{11},
\qquad
A = (\log r)_1,
\qquad
B = (\log G)_1,
\qquad
T = \wh{a}^{pq}\bg_{pq},
\]
and put $\eta_{pq} = b_{pq;1}$. The first-derivative identity \eqref{eq:first-deriv} cancels the two gradient-square terms in the quotient rule. Using \eqref{eq:b-evol}, we obtain
\begin{equation}\label{eq:Pw-exact}
\cP w = -\frac{S_{11}}{S\lambda} - \frac{\wh{a}^{pq}b_{11;pq}}{\lambda} - \frac{1}{\lambda} + \frac{1}{u} + \frac{\wh{a}^{pq}u_{pq}}{u} .
\end{equation}
Since $u_{pq} = b_{pq} - u\bg_{pq}$, the homogeneity identity \eqref{eq:homog} gives
\begin{equation}\label{eq:trace-term}
\frac{\wh{a}^{pq}u_{pq}}{u} = \frac{k}{u} - T .
\end{equation}
The terms $\beta/S$ have already canceled between $(\log b_{11})_t/S$ and $(\log u)_t/S$ in \eqref{eq:Pw-exact}.

The complete product rule for $S = e^{\gamma}r^{\alpha}G$ is
\begin{equation}\label{eq:S11}
\begin{aligned}
\frac{S_{11}}{S} = \;& \frac{\dot{G}^{pq}}{G}b_{pq;11} + \frac{\ddot{G}^{pq,rs}}{G}b_{pq;1}b_{rs;1} + \alpha\frac{r_{11}}{r} + \alpha(\alpha-1)A^2 + 2\alpha AB \\
& + \gamma_{11} + \gamma_1^2 + 2\gamma_1(\alpha A + B).
\end{aligned}
\end{equation}
Define
\begin{equation}\label{eq:Q3}
Q_3 := \frac{\ddot{G}^{pq,rs}b_{pq;1}b_{rs;1}}{\lambda G} .
\end{equation}
Because $\dot{G}/G = -\wh{a}$, the commutation formula \eqref{eq:commutation} and \eqref{eq:homog} yield
\begin{equation}\label{eq:commuted}
\frac{\wh{a}^{pq}(b_{pq;11} - b_{11;pq})}{\lambda} = \frac{k}{\lambda} - T .
\end{equation}
Substituting \eqref{eq:trace-term}--\eqref{eq:commuted} into \eqref{eq:Pw-exact} gives the exact identity
\begin{equation}\label{eq:Pw-full}
\cP w = \frac{k-1}{\lambda} + \frac{k+1}{u} - 2T - Q_3 - \frac{\alpha}{\lambda}\Big\{ \frac{r_{11}}{r} + (\alpha-1)A^2 + 2AB \Big\} + \wh{E}_f,
\end{equation}
where the complete angular-data error is
\begin{equation}\label{eq:Ef}
\wh{E}_f = -\frac{1}{\lambda}\big\{ \gamma_{11} + \gamma_1^2 + 2\gamma_1(\alpha A + B) \big\}.
\end{equation}
In particular, the trace term has the favorable sign $-2T$.

\medskip\noindent\emph{Step 3: the radial term and the combined quadratic form.}
We first translate the radial Hessian into the fixed spherical coordinates. Let $\nabla^M$ and $\Gamma^{M,q}_{ij}$ denote the connection and its Christoffel symbols for the induced metric on $M_t$. The identity $X_i = b_{ij}e_j$ implies, in spherical normal coordinates at the touching point,
\begin{equation}\label{eq:christoffel}
\Gamma^{M,q}_{11} b_{q\ell} = b_{1\ell;1} .
\end{equation}
Only this diagonal case is needed below; for repeated lower indices the remaining symmetry terms in the Christoffel symbols cancel by Codazzi.
By Codazzi, \eqref{eq:first-deriv}, and $r_i = u_j b_{ji}/r$, it follows that
\begin{equation}\label{eq:gamma11}
\Gamma^{M,q}_{11} r_q = \frac{\lambda \abs{\bn u}^2}{u r} .
\end{equation}
Since $X_{;1} = \lambda E_1$, where $E_1$ is the corresponding intrinsic unit principal direction, the scalar Hessian formula gives
\begin{equation}\label{eq:r11}
\frac{r_{11}}{r} = \frac{\lambda^2}{r}\, \nabla^{M,2}r(E_1, E_1) + \frac{\lambda \abs{\bn u}^2}{u r^2} .
\end{equation}
On the other hand, the Euclidean radial Hessian identity is
\begin{equation}\label{eq:radial-hessian}
\nabla^{M,2}r(E_1, E_1) = \frac{1}{r}\Big\{ -\frac{u}{\lambda} + 1 - (E_1 r)^2 \Big\}.
\end{equation}
Then the $C^1$ estimate gives
\[
1 - (E_1 r)^2 \ge \frac{u^2}{r^2} \ge \frac{u_{-}^2}{r_{+}^2} .
\]
Therefore, after discarding the nonpositive last term generated by \eqref{eq:r11},
\begin{equation}\label{eq:radial-term}
-\frac{\alpha}{\lambda}\, \frac{r_{11}}{r} \le -c_{*}\lambda + C,
\qquad
c_{*} := \frac{\alpha u_{-}^2}{r_{+}^4} > 0 .
\end{equation}

It remains essential to keep the inverse-concavity and mixed radial terms together. The function $G^{-1/k}$ is concave on the positive cone (see \eqref{eq:tildeF}). Thus
\[
0 \ge D^2\big(G^{-1/k}\big)[\eta, \eta] = -\frac{1}{k}G^{-1/k-1}D^2G[\eta, \eta] + \frac{k+1}{k^2}G^{-1/k-2}\big(DG[\eta]\big)^2 .
\]
Since $DG[\eta] = G_1 = GB$, this proves
\begin{equation}\label{eq:Q3-bound}
Q_3 \ge \frac{k+1}{k\lambda}\, B^2 .
\end{equation}
Define the combined form
\begin{equation}\label{eq:Qc}
Q_c := Q_3 + \frac{2\alpha AB + \alpha(\alpha-1)A^2}{\lambda} .
\end{equation}
Completing the square only once gives
\begin{equation}\label{eq:Qc-square}
Q_c \ge \frac{k+1}{k\lambda}\Big( B + \frac{\alpha k}{k+1}A \Big)^2 + \frac{\alpha(\alpha-k-1)}{(k+1)\lambda}\, A^2 \ge 0 .
\end{equation}
Here the last inequality is exactly where $\alpha \ge k+1$ is used. On the branch $\lambda \ge 1$, the zero-order terms in \eqref{eq:Pw-full} are bounded above, while $T \ge 0$. Combining \eqref{eq:Pw-full}, \eqref{eq:radial-term}, and \eqref{eq:Qc}, we obtain the core inequality
\begin{equation}\label{eq:core-ineq}
0 \le \cP w \le -c_{*}\lambda + C - Q_c + \wh{E}_f .
\end{equation}

\medskip\noindent\emph{Step 4: absorption of the data error.}
Differentiating $r^2 = u^2 + \abs{\bn u}^2$ and using $u_{ij} = b_{ij} - u\bg_{ij}$ gives
\begin{equation}\label{eq:r1-ids}
r_i = \frac{u_j b_{ji}}{r},
\qquad
r_1 = \frac{u_1 \lambda}{r},
\qquad
\frac{A}{\lambda} = \frac{u_1}{r^2} .
\end{equation}
Set
\[
Z := B + \frac{\alpha k}{k+1}A .
\]
Then
\begin{equation}\label{eq:Z-id}
\alpha A + B = Z + \frac{\alpha}{k+1}A .
\end{equation}
Since $f$ is positive and belongs to $C^2(\Sn)$,
\[
\norm{\bn\gamma}_{L^{\infty}} + \norm{\bn^2\gamma}_{L^{\infty}} \le C\big( f_{-}^{-1}, \norm{f}_{C^2} \big).
\]
Equations \eqref{eq:Ef}, \eqref{eq:r1-ids}, and \eqref{eq:Z-id} therefore give, for every $0 < \theta < 1$,
\begin{equation}\label{eq:Ef-absorb}
\wh{E}_f \le C + \frac{2\abs{\gamma_1}\abs{Z}}{\lambda} + \frac{2\alpha\abs{\gamma_1}\abs{A}}{(k+1)\lambda}
\le C_{\theta} + \theta\, \frac{k+1}{k\lambda}Z^2
\le C_{\theta} + \theta Q_c .
\end{equation}
This estimate also covers the critical value $\alpha = k+1$: although the last $A^2$ term in \eqref{eq:Qc-square} then vanishes, the remaining multiple of $A/\lambda$ in \eqref{eq:Z-id} is bounded by \eqref{eq:r1-ids}.

\medskip\noindent\emph{Step 5: the maximum-principle conclusion.}
Choose $\theta = 1/2$. Substitution of \eqref{eq:Ef-absorb} into \eqref{eq:core-ineq} gives
\[
0 \le \cP w \le -c_{*}\lambda + C - \frac{1}{2}Q_c \le -c_{*}\lambda + C .
\]
It follows that $\lambda \le C/c_{*}$ at an interior maximum. Hence
\[
\max_{\Sn \times [0, T']} W \le \max\Big\{ \max_{\Sn} W(\cdot, 0),\ \log(C/c_{*}) - \log u_{-} \Big\}.
\]
For every $(x, t)$ we therefore have $\Lambda(x, t) \le u_{+}\exp(\max W) \le C$. Letting $T' \uparrow T$ proves \eqref{eq:aniso-c2}. Notice in particular that the derivatives of $f$ enter only through \eqref{eq:Ef}; no uncontrolled trace of $\wh{a}^{ij}$ is created.
\end{proof}

\begin{corollary}[The isotropic case]\label{cor:iso-c2}
If $f$ is constant and $\alpha \ge k+1$, then every smooth strictly convex solution satisfying the $C^0$ and $C^1$ bounds obeys
\begin{equation}\label{eq:iso-c2-bound}
\lambda_{\max}(b_{ij}) \le C .
\end{equation}
\end{corollary}

\begin{proof}
This is the constant-data case of Proposition \ref{prop:aniso-c2}. If $f$ is constant, then $\gamma = \log f$ is constant, so $\wh{E}_f \equiv 0$ by \eqref{eq:Ef} and the absorption in Step 4 is unnecessary: \eqref{eq:core-ineq} already gives $0 \le \cP w \le -c_{*}\lambda + C$ at the maximum. This recovers, in the present normalization, the isotropic estimate of \cite[Lemma 3.7]{LSW20}.
\end{proof}

\begin{corollary}[Two-sided curvature bound]\label{cor:curv-bound}
There exists $C \ge 1$ such that
\begin{equation}\label{eq:curv-bound}
C^{-1} \le \kappa_i(\cdot, t) \le C
\end{equation}
for every $i = 1, \dots, n$ and every time in the existence interval.
\end{corollary}

\begin{proof}
Proposition \ref{prop:aniso-c2} gives $\kappa_i \ge C^{-1}$. Since all curvatures are positive, each monomial in $\sigma_k$ is positive. For a fixed $i$, expansion along the index $i$ gives
\[
\sigma_k(\kappa) \ge \kappa_i\,\sigma_{k-1}(\kappa \,|\, i),
\]
and the uniform lower bound for the remaining curvatures gives $\sigma_{k-1}(\kappa \,|\, i) \ge c > 0$. Corollary \ref{cor:speed} then gives
\[
C \ge \sigma_k(\kappa) \ge c\kappa_i,
\]
and hence $\kappa_i \le C$.
\end{proof}

\section{Long-time existence and higher regularity}\label{sec:reg}

For a smooth strictly convex initial support function $u_0$, the matrix $b[u_0] = \bn^2 u_0 + u_0\bg$ is positive definite. By \eqref{eq:Gdot-negative}, equation \eqref{eq:support-eq} is strictly parabolic in a neighborhood of $u_0$. Standard fully nonlinear parabolic theory (see, for instance, \cite{Ger06}) therefore gives a unique short-time solution. Strict convexity is an open condition, and the solution can be continued as long as the eigenvalues of $b$ stay in a compact subset of the positive cone and a fixed spatial H\"older norm remains bounded.

\begin{proposition}[Higher estimates]\label{prop:higher}
For every integer $m \ge 0$ there is a constant $C_m$ such that
\begin{equation}\label{eq:higher}
\norm{r(\cdot, t)}_{C^m(\Sn)} + \norm{u(\cdot, t)}_{C^m(\Sn)} \le C_m
\end{equation}
for all times in the maximal existence interval.
\end{proposition}

\begin{proof}
By the two-sided curvature estimate \eqref{eq:curv-bound} and the $C^0$--$C^1$ estimates of Lemmas \ref{lem:r-bound} and \ref{lem:u-bound}, the eigenvalues of the curvature-radius tensor $b = \bn^2 u + u\bg$ stay in a compact subset of the positive cone, uniformly in time.

We work with the support-function equation \eqref{eq:support-eq}, which we write as a fully nonlinear parabolic equation on the fixed sphere:
\begin{equation}\label{eq:Fop}
u_t = \cF(\bn^2 u, \bn u, u, x),
\qquad
\cF(A, z, s, x) := -f(x)\big( s^2 + \abs{z}^2 \big)^{\alpha/2} G\big( A + s\bg \big) + \beta s .
\end{equation}
The concavity of $G^{-1/k}$, already used in \eqref{eq:Q3-bound}, gives
\begin{equation}\label{eq:G-convex}
D^2 G(b)[\eta, \eta] \ge \frac{k+1}{k\, G(b)}\big( DG(b)[\eta] \big)^2 \ge 0
\end{equation}
for every symmetric tensor $\eta$; hence $G$ is convex in $b$, and $\cF(A, z, s, x)$ is concave in $A$ for every fixed $(z, s, x)$. Moreover,
\begin{equation}\label{eq:F-elliptic}
\frac{\partial \cF}{\partial A_{ij}} = -f(x)\, r^{\alpha}\, \dot{G}^{ij},
\end{equation}
so \eqref{eq:Gdot-negative}, Corollary \ref{cor:curv-bound}, and Lemmas \ref{lem:r-bound}--\ref{lem:u-bound} show that \eqref{eq:Fop} is uniformly parabolic, with ellipticity constants independent of $t$. Finally, $\abs{u_t} = \abs{u p} \le C$ on every compact existence interval by Proposition \ref{prop:residual-bounds}.

We use the parabolic Evans--Krylov theorem in the following form (see \cite[Chapter 5]{Kry87} and \cite[Chapter 14]{Lie96}). Let $v$ be a smooth solution of $v_t = \cG(\bn^2 v, \bn v, v, x)$ on $\Sn \times [t_0 - \tau, t_0]$, where $\cG$ is smooth, uniformly elliptic in its Hessian argument, and concave in its Hessian argument. Then there are $\theta \in (0, 1)$ and $C$, depending only on $\tau$, the ellipticity constants, the bounds for $\cG$ and its derivatives on the range of $(\bn^2 v, \bn v, v)$, and the spatial bounds $\sup_{s \in [t_0-\tau,\, t_0]}\norm{v(\cdot, s)}_{C^2(\Sn)}$ and $\norm{v_t}_{L^{\infty}}$, but not on $t_0$, such that
\[
\norm{v}_{C^{2+\theta,\, 1+\theta/2}(\Sn \times [t_0 - \tau/2,\, t_0])} \le C ,
\]
where $C^{2+\theta,\, 1+\theta/2}$ denotes the standard parabolic H\"older space with exponent $\theta$ in the spatial variables and $\theta/2$ in the time variable. The dependence on the separation $\tau$ between the two cylinders is intrinsic to interior estimates; in the application below, $\tau$ is fixed once and for all.

The previously established estimates place the arguments of $\cF$ in a fixed compact subset of its admissible domain: the eigenvalues of $b = \bn^2 u + u\bg$ lie in a compact subset of the positive cone, and $r$, $u$, and $\abs{\bn u}$ are uniformly bounded. Consequently, the structural derivatives of $\cF$ are uniformly bounded there, and its ellipticity constants are independent of time; the concavity in $A$ and the uniform ellipticity are verified above. The bounds for $\abs{\bn^2 u}$ and $\abs{u_t}$ also control the lower-order arguments $(\bn u, u)$ in the parabolic H\"older topology: they give spatial Lipschitz control of $\bn u$ and, by interpolation, temporal H\"older control with exponent $1/2$.

By short-time existence there is $\tau > 0$ such that the solution is smooth on $[0, 2\tau]$, with estimates depending on the initial datum. For every $t \ge 2\tau$ in the maximal existence interval, the backward cylinder $\Sn \times [t-\tau, t]$ is contained in the existence interval, and the Evans--Krylov theorem applied on it gives
\begin{equation}\label{eq:ek}
\norm{u}_{C^{2+\theta,\, 1+\theta/2}(\Sn \times [t-\tau/2,\, t])} \le C
\end{equation}
with $C$ and $\theta \in (0, 1)$ depending on $\tau$ but independent of $t$. Once \eqref{eq:ek} is established, the derivatives of $\cF$ evaluated along the solution --- including $\partial\cF/\partial A_{ij} = -f(x)r^{\alpha}\dot{G}^{ij}$, which involves the Hessian of $u$ --- are uniformly H\"older continuous in the parabolic metric. The equations obtained by differentiating \eqref{eq:Fop} in space and in time are therefore linear uniformly parabolic equations with uniformly H\"older continuous coefficients, since $f$ is smooth. Applying the parabolic Schauder estimates \cite[Chapter 4]{Lie96} on nested backward cylinders and evaluating at their terminal time gives
\[
\norm{u(\cdot, t)}_{C^{m}(\Sn)} \le C_m
\]
for every $m$, uniformly in $t$. Together with the short-time interval $[0, 2\tau]$, this proves the support-function part of \eqref{eq:higher}.

It remains to transfer the estimates to the radial function $r(\xi, t)$, which is a function of the radial variable $\xi$ rather than the normal variable $x$ (cf. Remark \ref{rem:angle}). Define
\[
\Xi_t(x) = \frac{X(x, t)}{\abs{X(x, t)}} .
\]
From $X = ux + \bn u$ one computes
\[
D\Xi_t = \frac{1}{r}\big( I - \xi \otimes \xi \big) b .
\]
For $y \perp x$, writing $\xi = \langle x, \xi\rangle x + \xi^{T}$ with $\xi^{T} \perp x$ and using $\abs{\langle \xi, y\rangle} = \abs{\langle \xi^{T}, y\rangle} \le (1 - \langle x, \xi\rangle^2)^{1/2}\abs{y}$, one obtains
\[
\abs{(I - \xi\otimes\xi)y}^2 = \abs{y}^2 - \langle \xi, y\rangle^2 \ge \langle x, \xi\rangle^2 \abs{y}^2 = \frac{u^2}{r^2}\abs{y}^2 .
\]
Consequently, with $\ell > 0$ the uniform lower bound for the eigenvalues of $b$ provided by Corollary \ref{cor:curv-bound},
\[
\abs{D\Xi_t(v)} \ge \frac{u}{r^2}\, \abs{bv} \ge \frac{u_{-}\, \ell}{r_{+}^2}\, \abs{v}
\qquad \text{for all } v \in T_x\Sn .
\]
Thus $\Xi_t : \Sn \to \Sn$ is a diffeomorphism with a uniformly nondegenerate differential: it is bijective because $M_t$ is star-shaped with respect to the origin, and its differential is nondegenerate by the estimate above. Since $\Xi_t = (ux + \bn u)(u^2 + \abs{\bn u}^2)^{-1/2}$ involves the first derivatives of $u$, a $C^m$ estimate for $\Xi_t$ uses the $C^{m+1}$ norm of $u$; the estimates for $u$ having been established at every order, this gives uniform $C^m$ estimates for $\Xi_t$ for every $m$. Differentiating the inverse-map identity then gives uniform $C^m$ estimates for $\Xi_t^{-1}$, and the composition
\[
r(\cdot, t) = \big( u^2 + \abs{\bn u}^2 \big)^{1/2} \circ\, \Xi_t^{-1}
\]
yields the claimed estimates for the radial function. For the continuation argument in Proposition \ref{prop:lte}, only the support-function estimates are used.
\end{proof}

\begin{proposition}[Long-time existence]\label{prop:lte}
The maximal smooth strictly convex existence time is infinite.
\end{proposition}

\begin{proof}
Lemmas \ref{lem:r-bound} and \ref{lem:u-bound}, Corollary \ref{cor:speed}, Proposition \ref{prop:aniso-c2}, Corollary \ref{cor:curv-bound}, and Proposition \ref{prop:higher} give time-independent $C^m$ bounds for every $m$. Suppose $T_{\max} < \infty$. Equation \eqref{eq:Fop} expresses $u_t$ as a smooth function of $\bn^2 u$, $\bn u$, $u$, and $x$, so the uniform estimates give
\begin{equation}\label{eq:ut-uniform}
\norm{u_t(\cdot, t)}_{C^m(\Sn)} \le C_m
\qquad \text{for every } m \text{ and all } t < T_{\max}.
\end{equation}
Consequently
\begin{equation}\label{eq:lipschitz}
\norm{u(\cdot, t) - u(\cdot, s)}_{C^m(\Sn)} \le C_m \abs{t - s},
\end{equation}
so $u(\cdot, t)$ converges in $C^{\infty}(\Sn)$ as $t \uparrow T_{\max}$ to a positive support function $u_{*}$ with $\bn^2 u_{*} + u_{*}\bg > 0$, the positivity being preserved by the uniform curvature bounds. Short-time existence from $u_{*}$ extends the solution past $T_{\max}$, contradicting maximality.
\end{proof}

This proves the long-time assertion in Theorem \ref{thm:main}.

\section{Exponential convergence to the soliton}\label{sec:conv}

In this section we prove the exponential convergence of the normalized flow to the stationary solution. The key quantity is the relative residual $p = \partial_t\log u$ introduced in Section \ref{subsec:residual}, and
\[
E(t) := \norm{p(\cdot, t)}_{L^{\infty}(\Sn)} .
\]
By Proposition \ref{prop:lte} the flow exists smoothly and remains strictly convex for all time, so Proposition \ref{prop:residual-bounds} and Remark \ref{rem:sharper-rate} apply on $[0, \infty)$: with
\[
\omega = \beta(\alpha - k - 1) > 0
\]
we have
\begin{equation}\label{eq:exp-rate}
E(t) \le C e^{-\omega t},
\qquad t \ge 0 ,
\end{equation}
where $C$ depends only on the initial datum and the fixed data. Note that the exponent $\omega$ is independent of the initial datum, and no curvature estimate enters its definition.

\begin{proof}[Proof of the convergence assertion in Theorem \ref{thm:main}]
Since $p = \partial_t\log u$, estimate \eqref{eq:exp-rate} gives
\begin{equation}\label{eq:logu-rate}
\norm{\partial_t\log u(\cdot, t)}_{L^{\infty}(\Sn)} \le C e^{-\omega t}.
\end{equation}
Consequently, for $t_2 > t_1 \ge 0$,
\[
\norm{\log u(\cdot, t_2) - \log u(\cdot, t_1)}_{L^{\infty}(\Sn)}
\le \int_{t_1}^{t_2} \norm{\partial_t\log u(\cdot, s)}_{L^{\infty}(\Sn)}\, ds
\le \frac{C}{\omega}\, e^{-\omega t_1}.
\]
Thus $\log u(\cdot, t)$ is Cauchy in $C^0(\Sn)$. There is a continuous function $u_{\infty} > 0$ such that
\begin{equation}\label{eq:c0-conv}
\norm{u(\cdot, t) - u_{\infty}}_{C^0(\Sn)} \le C e^{-\omega t}.
\end{equation}
Here the positivity and the passage from $\log u$ to $u$ follow from the uniform two-sided $C^0$ bounds in Lemmas \ref{lem:r-bound} and \ref{lem:u-bound}.

We next upgrade this convergence to all derivatives. By Proposition \ref{prop:higher} and equation \eqref{eq:support-eq}, which converts the established spatial estimates into time-derivative bounds, the coefficients $a^{ij}$, $d^i$, and $c$ of the linear equation $\cC[p] = 0$ in \eqref{eq:Cp} have uniform space-time H\"older bounds of every order on unit time cylinders, and $a^{ij}$ is uniformly positive definite. For every integer $m \ge 0$, the interior parabolic Schauder estimates on the translated time cylinders
\[
\Sn \times [t-1, t+1],
\qquad t \ge 1,
\]
and a standard bootstrap give
\[
\norm{p(\cdot, t)}_{C^m(\Sn)} \le C_m \norm{p}_{L^{\infty}(\Sn \times [t-1, t+1])}.
\]
Using \eqref{eq:exp-rate} on the right-hand side, we obtain
\begin{equation}\label{eq:p-cm}
\norm{p(\cdot, t)}_{C^m(\Sn)} \le C_m e^{-\omega t},
\qquad t \ge 1 .
\end{equation}
Since $u_t = up$, the product estimates, Proposition \ref{prop:higher}, and \eqref{eq:p-cm} imply
\begin{equation}\label{eq:ut-cm}
\norm{u_t(\cdot, t)}_{C^m(\Sn)} \le C_m e^{-\omega t},
\qquad t \ge 1 .
\end{equation}
Therefore, for $t_2 > t_1 \ge 1$,
\[
\norm{u(\cdot, t_2) - u(\cdot, t_1)}_{C^m(\Sn)} \le \int_{t_1}^{t_2} \norm{u_t(\cdot, s)}_{C^m(\Sn)}\, ds \le C_m e^{-\omega t_1}.
\]
It follows that $u(\cdot, t)$ converges to the same function $u_{\infty}$ in $C^m(\Sn)$ for every $m$, and
\begin{equation}\label{eq:u-cm}
\norm{u(\cdot, t) - u_{\infty}}_{C^m(\Sn)} \le C_m e^{-\omega t}.
\end{equation}
The estimates above are stated for $t \ge 1$; enlarging $C_m$ to include the compact interval $[0, 1]$ proves \eqref{eq:thm-rate} for all $t \ge 0$, and the short-time smoothness already justifies this step. In particular, $u_{\infty}$ is smooth. The uniform two-sided curvature bound \eqref{eq:curv-bound} shows that the limit hypersurface remains strictly convex. Moreover, the inverse Gauss map parametrization \eqref{eq:position} and \eqref{eq:u-cm} show that $M_t$ converges exponentially to this hypersurface in $C^{\infty}$.

Finally, $p \to 0$ and $u \to u_{\infty}$ in $C^{\infty}$. Passing to the limit in
\[
p = \beta - \frac{f r^{\alpha}G(b[u])}{u}
\]
gives
\[
f(x) r_{\infty}^{\alpha} \sigma_k(\kappa[u_{\infty}]) = \beta u_{\infty}.
\]
Thus the limit is a smooth strictly convex stationary solution of \eqref{eq:stationary}, completing the proof.
\end{proof}

\begin{remark}[Optimality of the exponent]\label{rem:optimal}
The exponent $\omega = \beta(\alpha-k-1)$ in \eqref{eq:thm-rate} cannot be replaced by a larger one in general, even for a fixed positive angular function $f$. Indeed, let $u_{\infty}$ be the stationary support function and consider the homothetic family $u(x, t) = a(t)u_{\infty}(x)$. By the homogeneity of $r$ and $G$, the support-function equation \eqref{eq:support-eq} reduces to
\[
a' = \beta a\big( 1 - a^{m} \big),
\qquad m = \alpha - k - 1,
\]
whose solution with $a(0) = a_0 > 0$ is
\[
a(t) = \Big[ 1 + \big( a_0^{-m} - 1 \big) e^{-m\beta t} \Big]^{-1/m} .
\]
For $a_0 \neq 1$,
\[
a(t) - 1 = -\frac{1}{m}\big( a_0^{-m} - 1 \big)\, e^{-m\beta t} + O\big( e^{-2m\beta t} \big),
\]
so homothetic copies of the stationary solution already realize the decay rate $e^{-\omega t}$.
\end{remark}

\section{Uniqueness of the stationary solution}\label{sec:unique}

\begin{proposition}\label{prop:unique}
If $\alpha > k+1$, equation \eqref{eq:stationary} has at most one smooth strictly convex positive solution.
\end{proposition}

\begin{proof}
Suppose that $u$ and $v$ are two solutions and let
\[
C = \max_{\Sn} \frac{u}{v} .
\]
At a maximum point $x_0$ of $u/v$,
\[
u = Cv,
\qquad
\bn u = C\bn v,
\qquad
b[u] \le C b[v].
\]
Consequently $r[u] = Cr[v]$ and
\[
\sigma_k(\kappa[u]) \ge C^{-k}\sigma_k(\kappa[v]).
\]
Write $S[u] := f(x) r[u]^{\alpha}\sigma_k(\kappa[u])$ for the contracting speed factor. Using the stationary equations at $x_0$ gives
\[
\beta Cv = S[u] \ge C^{\alpha - k} S[v] = C^{\alpha-k}\beta v .
\]
Thus $C \ge C^{\alpha-k}$. Since $\alpha - k > 1$, this rules out $C > 1$. Therefore $u \le v$. Interchanging $u$ and $v$ proves $u = v$.
\end{proof}

\begin{remark}\label{rem:dl-uniqueness}
Proposition \ref{prop:unique} is consistent with the general uniqueness criterion of Ding--Li \cite[Theorem 6.1]{DL25} for star-shaped solutions of $F(\kappa) = G(X, \nu)$ with $\partial_{\rho}(\rho G) < 0$. Writing \eqref{eq:stationary} in the rooted form
\[
\sigma_k(\kappa)^{1/k} = \Big( \frac{\beta}{f(\nu)} \Big)^{1/k} \langle X, \nu\rangle^{1/k} \abs{X}^{-\alpha/k},
\]
the right-hand side is positive on the positive-support domain $\langle X, \nu\rangle > 0$, which contains every convex hypersurface enclosing the origin, and
\[
\partial_{\rho}\big( \rho G(\rho\xi, \nu) \big) = -\frac{\alpha - k - 1}{k}\, G(\rho\xi, \nu) < 0
\qquad (\rho = \abs{X})
\]
exactly when $\alpha > k+1$. The short support-function proof above keeps the present paper self-contained.
\end{remark}

\begin{proof}[Proof of Corollary \ref{cor:main}]
Take $\alpha = n+1-q > k+1$ and $f = \beta(n+1)\phi$, and evolve any smooth strictly convex initial hypersurface enclosing the origin under \eqref{eq:normflow}. Theorem \ref{thm:main} gives a smooth strictly convex limit satisfying \eqref{eq:stationary}, which with this choice of $f$ and $\alpha$ is precisely \eqref{eq:cor-pde}. Equivalently, the limit body $K_{\infty}$ satisfies $dC_{k,q}(K_{\infty}, \cdot) = \phi(x)\,dx$ by \eqref{eq:Ckq-density}. Uniqueness follows from Proposition \ref{prop:unique}.
\end{proof}

\begin{corollary}[Constant-data case]\label{cor:isotropic}
Let $f \equiv f_0 > 0$, $1 \le k \le n$, and $\alpha > k+1$. Then for every smooth, closed, strictly convex initial hypersurface enclosing the origin, the normalized flow \eqref{eq:normflow} exists for all time and converges exponentially in $C^{\infty}$ to the round sphere of radius
\begin{equation}\label{eq:sphere-radius}
R_{*} = f_0^{-1/(\alpha-k-1)} .
\end{equation}
\end{corollary}

\begin{proof}
This is the constant-data specialization of Theorem \ref{thm:main}. The stationary limit is unique by Proposition \ref{prop:unique}. Rotating the limit gives another solution of the same equation, hence uniqueness makes it rotationally invariant. Substitution of a sphere of radius $R$ into \eqref{eq:stationary} gives $f_0 R^{\alpha-k-1}\binom{n}{k} = \beta$, and since $\beta = \binom{n}{k}$ this gives \eqref{eq:sphere-radius}. This is the strictly supercritical, uniformly convex case of \cite{LSW20}.
\end{proof}

\begin{remark}[The critical exponent]\label{rem:critical}
The hypothesis $\alpha > k+1$ enters the argument at several distinct points: the uniform radial barriers of Lemma \ref{lem:r-bound} use $m = \alpha - k - 1 > 0$ in \eqref{eq:rmin}--\eqref{eq:rmax}; the absorption in the $C^2$ estimate requires only $\alpha \ge k+1$ (see \eqref{eq:Qc-square}); the exponential decay of the residual requires the strict inequality, since the reaction coefficient in \eqref{eq:residual-eq} is $c = (\alpha-k-1)S/u$; and the uniqueness comparison in Proposition \ref{prop:unique} uses $\alpha - k > 1$. Thus the principal-radius estimate is conditional and remains valid at the critical exponent, whereas the complete theorem relies on strict supercriticality for uniform scale control, residual decay, and uniqueness. At the critical exponent the obstruction is genuine rather than technical. Indeed, when $\alpha = k+1$ and $f \equiv f_0$ is constant, a round sphere of radius $R(t)$ evolves under \eqref{eq:normflow} according to
\begin{equation}\label{eq:sphere-ode}
R' = -f_0 R^{\alpha}\binom{n}{k} R^{-k} + \beta R = \beta\big( 1 - f_0 R^{\alpha-k-1} \big) R,
\end{equation}
using $\beta = \binom{n}{k}$. At $\alpha = k+1$ this reduces to $R' = \beta(1 - f_0)R$: if $f_0 = 1$, every radius is stationary, so the stationary equation has a continuum of solutions and uniqueness fails; if $f_0 \neq 1$, the radius drifts exponentially towards $0$ or $\infty$, so convergence to a stationary hypersurface fails. This explains why the conditional $C^2$ estimate at the critical exponent in Proposition \ref{prop:aniso-c2} does not imply the conclusions of Theorem \ref{thm:main} there.
\end{remark}

{\bf AI usage.}
The authors used GPT-5.6 Sol and KIMI K3 to assist with routine computations, and to help identify potential issues in the arguments. In particular, AI tools found two papers closely related to this work, namely references \cite{BIS21} and \cite{DL25}.  The authors have thoroughly checked all mathematical derivations and proofs and take full responsibility for the entire content of this manuscript.

{\bf Acknowledgements.}
W. Sheng was partially supported by National Key R$\&$D Program of China (No. 2022YFA1005500) and Natural Science Foundation of China under Grant No. 12571063.



\begin{thebibliography}{99}

\bibitem{And94} B. Andrews, Contraction of convex hypersurfaces in Euclidean space, \emph{Calc. Var. Partial Differential Equations} \textbf{2} (1994), 151--171.

\bibitem{BIS21} P. Bryan, M. N. Ivaki and J. Scheuer, Parabolic approaches to curvature equations, \emph{Nonlinear Anal.} \textbf{203} (2021), Paper No. 112174, 24 pp.

\bibitem{CNS85} L. Caffarelli, L. Nirenberg and J. Spruck, The Dirichlet problem for nonlinear second-order elliptic equations III: functions of the eigenvalues of the Hessian, \emph{Acta Math.} \textbf{155} (1985), 261--301.

\bibitem{DL23} S. Ding and G. Li, A class of inverse curvature flows and $L_p$ dual Christoffel--Minkowski problem, \emph{Trans. Amer. Math. Soc.} \textbf{376} (2023), 697--752.

\bibitem{DL25} S. Ding and G. Li, A flow method for curvature equations, \emph{Nonlinear Anal.} \textbf{261} (2025), Paper No. 113873.

\bibitem{Ger06} C. Gerhardt, \emph{Curvature Problems}, International Press, Somerville, MA, 2006.

\bibitem{GLL12} P. Guan, J. Li and Y. Y. Li, Hypersurfaces of prescribed curvature measure, \emph{Duke Math. J.} \textbf{161} (2012), 1927--1942.

\bibitem{GLM09} P. Guan, C.-S. Lin and X.-N. Ma, The existence of convex body with prescribed curvature measures, \emph{Int. Math. Res. Not. IMRN} (2009), no. 11, 1947--1975.

\bibitem{GM03} P. Guan and X.-N. Ma, The Christoffel--Minkowski problem I: Convexity of solutions of a Hessian equation, \emph{Invent. Math.} \textbf{151} (2003), 553--577.

\bibitem{HLYZ16} Y. Huang, E. Lutwak, D. Yang and G. Zhang, Geometric measures in the dual Brunn--Minkowski theory and their associated Minkowski problems, \emph{Acta Math.} \textbf{216} (2016), 325--388.

\bibitem{Kry87} N. V. Krylov, \emph{Nonlinear Elliptic and Parabolic Equations of the Second Order}, Reidel, Dordrecht, 1987.

\bibitem{LSW20a} Q.-R. Li, W. Sheng and X.-J. Wang, Flow by Gauss curvature to the Aleksandrov and dual Minkowski problems, \emph{J. Eur. Math. Soc.} \textbf{22} (2020), 893--923.

\bibitem{LSW20} Q.-R. Li, W. Sheng and X.-J. Wang, Asymptotic convergence for a class of fully nonlinear curvature flows, \emph{J. Geom. Anal.} \textbf{30} (2020), 834--860.

\bibitem{LXZ22} H. Li, B. Xu and R. Zhang, Asymptotic convergence for a class of anisotropic curvature flows, \emph{J. Funct. Anal.} \textbf{282} (2022), Paper No. 109460, 34 pp.

\bibitem{Lie96} G. M. Lieberman, \emph{Second Order Parabolic Differential Equations}, World Scientific, Singapore, 1996.

\bibitem{Sch14} R. Schneider, \emph{Convex Bodies: The Brunn--Minkowski Theory}, 2nd ed., Cambridge University Press, Cambridge, 2014.

\bibitem{SY27} W. Sheng and J. Yang, Long time behavior of a class of non-homogeneous anisotropic fully nonlinear curvature flows, \emph{J. Funct. Anal.} \textbf{292} (2027), Paper No. 111668.

\end{thebibliography}
\end{document}